\documentclass[twoside,10pt,a4paper,intlimits]{amsart}

\usepackage[latin1]{inputenc}

\usepackage[centertags]{amsmath}%center-tags centered tags at split

\usepackage{amsmath}
\usepackage{amsthm}
\usepackage{amssymb}
\usepackage{amsfonts}
\usepackage{amsxtra}
\usepackage{bbm}
\usepackage{color}
\usepackage{xspace}
\usepackage{mathabx}
\makeatletter
\let\diameter\@undefined                        % undefine \proof
\makeatother

\usepackage{hyperref}
\usepackage{enumitem}

\usepackage{ifthen}
\usepackage{graphicx}

\newtheorem{theorem}{Theorem}[section]
\newtheorem{lemma}[theorem]{Lemma}
\newtheorem{proposition}[theorem]{Proposition}
\newtheorem{corollary}[theorem]{Corollary} 

\newtheorem{definition}[theorem]{Definition}

\newtheorem{assumption}[theorem]{Assumption}

\newtheorem{remark}[theorem]{Remark}

\definecolor{lightgrey}{rgb}{.7,.7,.7}

\newcommand{\hellgrau}[1]{\textcolor{lightgrey}{#1}}

\newcounter{notectr}
\newcommand{\note}[1]{\ifthenelse{\thenotectr=1}{\hellgrau{#1}}{}}
\DeclareMathOperator{\spn}{span}

\newcommand{\argmin}{\ensuremath{\operatorname{argmin}}}

\numberwithin{equation}{section}

\newcommand{\definedas}{\mathrel{:=}}
\newcommand{\energy}{\ensuremath\mathcal{J}}
\newcommand{\mcP}{\mathcal{P}}
\newcommand{\setR}{\mathbb{R}}
\newcommand{\setN}{\mathbb{N}}
\newcommand{\Rn}{\setR^n}
\newcommand{\Rm}{\setR^m}
\newcommand{\dual}[2]{\ensuremath{\left\langle#1,\,#2\right\rangle}}
\newcommand{\tria}{\ensuremath{\mathcal{T}}}
\newcommand{\VT}{\ensuremath{\mathbb{V}(\tria)}}
\newcommand{\VoT}{\ensuremath{\mathbb{V}_0(\tria)}}
\newcommand{\nodes}{\ensuremath{\mathcal{N}}}

\newcommand{\ie}{i.e.,\xspace}
\newcommand{\Frechet}{Fr{\'e}chet\xspace}
\renewcommand{\P}{\ensuremath{\mathbb{P}}}
\renewcommand{\vec}[1]{\boldsymbol{#1}}
\newcommand{\sides}{\ensuremath{\mathcal{S}}}

\newcommand{\convexhull}{\operatorname{conv\,hull}}
\newcommand{\extr}{\operatorname{extr}}

\providecommand{\settmp}[2]{{#1\{{#2}#1\}}}
\providecommand{\set}[1]{\settmp{}{#1}}

\providecommand{\abs}[1]{{\lvert{#1}\rvert}}

\providecommand{\support}{\ensuremath{\operatorname{supp}}}

\begin{document}

\title[Convex Hull Property for Minimisers]{Convex Hull Property and Maximum Principle for Finite Element
  Minimisers of\\ General Convex Functionals}

\author[L.~Diening]{Lars Diening}
\address{Lars Diening,
 Mathematisches Institut der Universit{\"a}t M{\"u}nchen,
 Theresienstrasse 39, D-80333 M{\"u}nchen, Germany
 }%
\urladdr{http://www.mathematik.uni-muenchen.de/~diening}
\email{Lars.Diening@mathematik.uni-muenchen.de}

\author[C.~Kreuzer]{Christian Kreuzer}
\address{Christian Kreuzer,
 Fakult{\"a}t f{\"u}r Mathematik,
 Ruhr-Universit{\"a}t Bochum,
 Universit{\"a}tsstrasse 150, D-44801 Bochum, Germany
 }%
\urladdr{http://www.ruhr-uni-bochum.de/ffm/Lehrstuehle/Kreuzer/}
\email{christan.kreuzer@rub.de}
% \thanks{The research of Christian Kreuzer was partially 
%   supported by the DFG research grant KR 3984/1-1.}

\author[S.~Schwarzacher]{Sebastian Schwarzacher}
\address{Sebastian Schwarzacher,
 Mathematisches Institut der Universit{\"a}t M{\"u}nchen,
 Theresienstrasse 39, D-80333 M{\"u}nchen, Germany
 }%
\urladdr{http://www.mathematik.uni-muenchen.de/~schwarz/}
\email{sebastian.scharzacher@mathematik.uni-muenchen.de}

\begin{abstract}
  The convex hull property is the natural generalization of maximum principles
  from scalar to vector valued functions. Maximum principles for
  finite element approximations are often crucial for the preservation of
  qualitative properties of the respective physical model. In this
  work we develop a convex hull property for
  $\P_1$ conforming finite elements on simplicial non-obtuse meshes.  
  The proof does not resort to linear structures of partial
  differential equations but directly addresses properties of the
  minimiser of a convex energy functional.
  Therefore, the result holds for very general nonlinear partial
  differential equations including e.g.
  the $p$-Laplacian and the mean curvature problem.
  In the case of scalar equations the introduce techniques can be used
  to prove standard discrete maximum principles for nonlinear
  problems. We conclude by proving a strong discrete convex hull property on
  strictly acute triangulations.  
\end{abstract}

\keywords{discrete maximum principle \and strong discrete maximum principle \and finite elements \and nonlinear pde
  \and mean curvature \and p-Laplace}
  \subjclass[2010]{65N30\and 35J60 \and 35J92 \and
    35J93}

\maketitle

%\tableofcontents

\section{Introduction}
\label{sec:introduction}

Let the bounded polyhedral domain $\Omega\subset \Rn$, $n \geq 2$, be
partitioned by a conforming simplicial triangulation $\tria$. We
denote by $\VT^m$, $m\in\setN$, the finite
element space of $m$-dimensional vector valued, continuous, and piece
wise affine functions over $\tria$ and by $\VoT^m\subset \VT^m$ the subspace 
of functions with zero boundary values. For $\vec{G}\in\VT^m$ we consider minimising
problems of the form
\begin{align}
  \label{eq:prob}
  \energy(\vec{V}):=\int_\Omega F\big(x,\abs{\nabla \vec{V} (x)}\big)\,{\rm d}x\to
  \min \qquad\text{in}~ \vec{G}+\VoT^m.
\end{align}
We suppose that \eqref{eq:prob} admits a unique minimiser and that
the mapping $\Omega\times \setR_\ge \ni(x,y)\mapsto
F(x,y)\in\setR$ is monotone in its second argument; 
for the precise assumptions see
Section~\ref{sec:energy-minimisation}.

The following theorem is the main result of this work.

\begin{theorem}[Convex hull property]\label{t:main}
  Let the triangulation $\tria$ be non-obtuse, \ie the angle between
  any two sides of any simplex $T\in\tria$ is less or equal $\pi/2$. 
  
  Then the unique minimiser $\vec{U}\in \vec{G}+\VoT^m$
  of \eqref{eq:prob}
  satisfies
  \begin{align}\label{eq:CHP}
    \vec{U}(\Omega)\subset
    \operatorname{conv\,hull}\big(\vec{U}(\partial\Omega)\big)
    =\operatorname{conv\,hull}\big(\vec{G}(\partial\Omega)\big).
    \tag{CHP}
  \end{align}
\end{theorem}
This is the discrete analogue of the so called convex hull property of
vector valued minimisers; see~\cite{OttLM98,BilFuc02}. Note that the
convex hull property is the generalization of the maximum principle to
the vector valued case. Indeed, if $m=1$, then the convex hull
property reads as
\begin{align*}
  U(\Omega) = [\min U(\partial \Omega), \max U(\partial \Omega)].
\end{align*}

The problem above is presented as a minimising problem. If the energy
$\energy$ is \Frechet differentiable, then~\eqref{eq:prob} becomes
equivalent to finding a solution $\vec{U}\in \vec{G}+\VoT$ of the
Euler-Lagrange system
\begin{align}
  \label{eq:DJ}
  \dual{ D\energy(\vec{U})}{\vec{H}}=\int_\Omega a\left(x,\abs{\nabla
      \vec{U}(x)}\right)\,\nabla \vec{U}\colon \nabla \vec{H}\,{\rm
    d}x =0\quad\text{for all $\vec{H}\in\VoT$,} 
\end{align}
where $a(x,t)=F_t(x,t)/t$. Consequently, Theorem \ref{t:main} applies
to a large class of non-linear partial differential equations.

The research on discrete
maximum principles can be traced back to the early paper of Ciarlet
and Raviart \cite{CiarletRaviart:73} in the nineteen seventies.
Provided a non-obtuse triangulation, they established a discrete
maximum principle for scalar valued linear second order elliptic
problems on a two dimensional domain.   This result
was generalized to three dimensional domains in \cite{KrizekQun:95}.
For the Laplacian, the above mesh condition can be replaced by a
weaker, so called Delaunay condition. This condition only requires
that the sum of any pair of angles opposite a common side is less or
equal $\pi$; see \cite{StrangFix:73,Letniowski:92}.  Dr{\u{a}}g{\u{a}}nescu, Dupont,
and Scott \cite{DragDupScott:04} analyzed the failure of the discrete
maximum principle when the Delaunay condition is violated.  However,
according to \cite{Santos:82,KorotovKrizekNeittaanmaki:01} a discrete
maximum principle may hold even if both angles in such a pair are
greater than $\pi/2$ introducing additional restrictions involving a
larger neighborhood.  Some of these concepts are generalized to
anisotropic diffusion problems in \cite{LiHuang:10,Huang:11}.

Recently, discrete maximum principles are proved for quasi-linear
partial differential equations \cite{KaratsonKorotov:05,WangZhang:11},
which means in the case of \eqref{eq:DJ} that $ 0<\lambda \le
a(x,t)\le \Lambda<\infty$ for all $x,t\in\Omega\times\setR$.
J{\"u}ngel and Unterreitner \cite{JunUnt:05} proved a discrete maximum principle for
partial differential equations with non-linear lowest order term.

All aforementioned approaches
are based on the ellipticity and continuity of the problems in
order to allow for a Hilbert space setting. In particular, beside
\cite{WangZhang:11,JunUnt:05}  they
rely on the inversion of so called $M$-Matrices. 
As a consequence the results are basically restricted to energy
functionals with quadratic growth.

In contrast, our
approach is based on a projection into a suitable convex set and direct
properties of the energy functional.  Therefore, it does not rely
on linear structures and directly applies to general convex
non-linear, vector valued problems. As a drawback, we need point-wise
(resp. element-wise) properties of the finite element functions and
therefore cannot weaken the mesh restrictions to a Delaunay type condition.

The plan of the paper is as follows.
In Section \ref{sec:problem-setting} we introduce the finite element
framework and provide precise assumptions on the energy functional
from \eqref{eq:prob}. 
In Section \ref{sec:main} we prove the discrete convex hull property, Theorem
\ref{t:main}. We close the article in 
Section \ref{sec:extensions} discussing several
applications and extensions of the introduced theory.
In particular, we prove a maximum principle for non-linear scalar partial
differential equations with non-positive right hand side and 
show that the $p$-Laplace, 
as well as the mean curvature problem fit into the presented theory.
Moreover, we show
a discrete maximum principle for non-linear problems with lower order
terms using mass-lumping. Finally, we define and verify a strong
convex hull property, which generalizes the strong maximum principle
of scalar valued problems.

\renewcommand{\atop}[2]{\ensuremath{\genfrac{}{}{0pt}{}{#1}{#2}}}

\section{Problem Setting}
\label{sec:problem-setting}
In this section we shall introduce the
finite element setting and formulate the
precise assumptions on the non-linear energy.
For the sake of presentation we denote the Euclidean norms in $\Rn$, $n\ge2$,
 $\Rm$, $m\in\setN$, as well as the absolute value in $\setR$ by~$\abs{\cdot}$. Furthermore,  for
 $\vec{A}=[A_{ij}]_{\atop{i=1,\ldots,n}{j=1,\ldots,m}},\vec{B}=[B_{ij}]_{\atop{i=1,\ldots,n}{j=1,\ldots,m}}\in\setR^{n\times
   m}$ we define  the inner product
$\vec{A}:\vec{B}\definedas\sum_{\atop{i=1,\ldots,n}{j=1,\ldots,m}}A_{ij}B_{ij}$
and denote also by $\abs{\vec{A}}$ the corresponding Frobenius norm of $\vec{A}$.

\subsection{Finite Element Framework}
\label{sec:fe}

Let $\tria$ be a conforming partition of the polyhedral domain
$\Omega\subset\setR^n$ into closed
$n$-simplexes. To be more precise, we have
\begin{align*}
  \bar \Omega = \bigcup \{T\mid T\in\tria\}
\end{align*}
 and the intersection of two different elements in $\tria$ is either empty or a complete
 $k$-dimensional sub-simplex, 
 $0\le k < n$. 
 Let $\P_1(T)$ be the space of affine linear functions on
 $T\in\tria$ and define the space of first order  Lagrange finite elements 
 by
 \begin{align}
   \label{eq:VT}
   \VT\definedas \left\{ V\in C(\bar\Omega)\colon
     V_{|T}\in\P_1(T)~\text{for all}~T\in\tria\right\}.
 \end{align}
 This space is spanned by the so called nodal Lagrange basis.
 To be more precise, let $\nodes$ be the set of 
 vertices of all elements in $\tria$ and let $\nodes_0\definedas
 \nodes\cap\Omega$ be the subset of vertices inside $\Omega$.
 For $z\in\nodes$ the corresponding Lagrange hat function $\phi_z$ is
 uniquely defined by $\phi_z(y)=\delta_{zy}$ for all $y\in\nodes$,
 where we make use of the Kronecker symbol, \ie $\delta_{zy}=1$
 whenever $z=y$ and $\delta_{zy}=0$ else. We denote the set of element
 sides  of $\tria$ by $\sides$. For two sides of  $S_1,S_2\in\sides$ of a
 simplex $T\in\tria$ we define the angle between $S_1$ and $S_2$ 
 by $\angle(S_1,S_2)\definedas\pi-\angle(\vec{n}_1,\vec{n}_2)$, where
 $\angle(\vec{n}_1,\vec{n}_2)$ is the angle in between the normals
 $\vec{n}_1$ and $\vec{n}_2$ pointing inside $T$ from $S_1$ and $S_2$,
 respectively.

 In the following we introduce non-obtuse meshes, which play a crucial
 role in the analysis of discrete 
 maximum principles; see e.g. 
 \cite{CiarletRaviart:73,KrizekQun:95,DragDupScott:04}.
 \begin{definition}[non-obtuseness]\label{df:non-obtuse}
   A simplex $T\in\tria$ is called non-obtuse 
   if the angles between any two of its sides 
   are less or equal $\frac\pi2$. 

   The conforming triangulation $\tria$ of $\Omega$ is called 
     non-obtuse if all simplexes $T\in\tria$ are non-obtuse.
 \end{definition}
 We remark that in higher
 dimensions $n\ge3$ the construction and refinement of non-obtuse
 partitions is  a non trivial task. The
 interested reader is e.g. 
 referred to \cite{KrizekQun:95,KorotovKrizek:2005,KorotovKrizek:2011}. 

 One can characterize non-obtuse
 meshes by means of the nodal Lagrange basis functions. Indeed, 
 observing that the gradient of a Lagrange basis function $\phi_z$,
 $z\in\nodes\cap T$  
 on a simplex $T\in\tria$ points in the direction of the normal of
 the side opposite $z$, 
 one can easily show that the following condition on the basis
 functions  is 
 equivalent to the geometric mesh property in Definition
 \ref{df:non-obtuse}; compare also with \cite{CiarletRaviart:73,KrizekQun:95}. 
 \begin{proposition}\label{p:non-obtuse}
  An $n$-simplex $T\in\tria$ is 
     non-obtuse if and only if
   \begin{align*}
     \nabla \phi_{z|T}\cdot\nabla\phi_{y|T}\le 0 \qquad\text{for
       all}~z,y\in T\cap\nodes~\text{with}~z\neq y.
   \end{align*}
 \end{proposition}
 
 A direct consequence is the following corollary.
\begin{corollary}
  \label{cor:non-obtuse} 
  Let $\tria$ be a conforming triangulation of $\Omega$. Then 
  $\tria$ is non-obtuse if and only if
  \begin{align*}
    \nabla \phi_z \cdot \nabla \phi_y \leq 0 \quad\text{a.\,e. in $\Omega$\quad
      for all}~z,y\in\nodes~\text{with}~z\neq y.
  \end{align*}
\end{corollary}

The subspace of functions in $\VT$ vanishing at the boundary
$\partial\Omega$ is denoted by 
\begin{align*}
  \VoT\definedas\left\{V\in\VT\colon
      V_{|\partial\Omega}=0\right\}=\spn\left\{\phi_z\colon z\in\nodes_0\right\}.
\end{align*}
For $m\in\setN$, the $m$-dimensional vector valued finite element
spaces corresponding 
to $\VT$ and $\VoT$
are given by 
\begin{align*}
  \VT^m\definedas \bigtimes_{i=1}^m\VT\qquad\text{and}\qquad
  \VoT^m\definedas \bigtimes_{i=1}^m\VoT, 
\end{align*}
respectively. Clearly, for $\vec{G}\in\VT^m$, the affine subspace of functions that coincide
with $\vec{G}$ on $\partial\Omega$ is $\vec{G}+\VoT^m$. For
$\omega\subset\bar\Omega$ we define $\vec{V}(\omega)\definedas
\{\vec{V}(x)\colon x\in\omega\}\subset\setR^m$ as the set of function values of
$\vec{V}$ on $\omega$.

% We emphasize that the introduced finite element spaces $\VT^m$ and
% $\VoT^m$ can be identified with finite dimensional $\setR$-vector
% spaces. Consequently they are Banach spaces and on $\VT^m$
% respectively $\VoT^m$ all norms are equivalent. For this reason we denote by
% $\norm{\cdot}:\VT^m\to\setR$ an abstract norm, which for particular
% problems may be identified with e.g. some Sobolev  norm; compare also
% with Section~\ref{sec:extensions} below.

\subsection{Energy Minimisation}
\label{sec:energy-minimisation}
In this section we formulate standard conditions on the energy
$\energy$ in order to
guarantee the unique solvability of
\eqref{eq:prob}. To this end we first ensure that $\energy(\vec{V})$
is well defined for all $\vec{V}\in\VT^m$ assuming that
$F\colon\Omega\times\setR_\ge\to\setR$ defines a Nemyckii operator 
$\mathcal{F}:\VT^m\to L^1(\Omega)$ setting
\begin{align*}
  (\mathcal{F}\vec{V})(x)\mapsto F(x,\abs{\nabla\vec{V}(x)}), \quad x\in\Omega.
\end{align*}
This can be achieved by the following conditions.

\begin{assumption}[Nemyckii Operator] \label{ass:Nemyckii} We assume  that
$F\colon\Omega\times\setR_\ge\to\setR$  
  \begin{enumerate}[leftmargin=1cm,itemsep=1ex,label=(\arabic{*})]
  \item\label{i:carat} is a {\em Carath{\'e}odory function}, \ie $F(\cdot,t)$ is
    measurable in $\Omega$ for each $t\in\setR_\ge$ and $F(x,\cdot)$
    is continuous on $\setR_\ge$ for almost every $x\in\Omega$;
  \item \label{i:growth}satisfies the {\em growth condition}
    \begin{align*}
      \abs{F(x,t)}\leq \alpha(x)+\gamma(t),\quad (x,t)\in\Omega\times
    \setR_\ge,
    \end{align*}
    for some $\alpha\in L^1(\Omega)$ and continuous
    $\gamma:\setR_\ge\to\setR_\ge$. 
  \end{enumerate}
\end{assumption}

\begin{proposition}\label{prop:cont}Suppose that
  $F\colon\Omega\times\setR_\ge\to\setR$ satisfies Assumptions
  \ref{ass:Nemyckii}. Then the energy functional $\energy:\VT^m\to
  L^1(\Omega)$ defined in \eqref{eq:prob} is continuous.
\end{proposition}
\begin{proof}
  Let $\{\vec{V}_k\}_{k\in\setN}\subset\VT^m$, $\vec{V}\in\VT^m$ such that
  $\vec{V}_k\to\vec{V}$ in $\VT^m$ as $k\to\infty$.
  The space $\VT^m$ is finite
  dimensional and hence all norms on $\VT^m$ are equivalent. 
  Therefore, we can conclude that the sequence
  $\{\vec{V}_k\}_{k\in\setN}$ is uniformly bounded in
  $L^\infty(\Omega)^m$. Assumption
  \ref{ass:Nemyckii}\ref{i:growth} thus implies for some $C>0$ that 
  $|F(\cdot,\vec{V}_k(\cdot))|\le \alpha(\cdot)+ C\in L^1(\Omega)$ and  
  from Assumption \ref{ass:Nemyckii}\ref{i:carat} it follows that
  $F(\cdot,\vec{V}_k(\cdot))\to F(\cdot,\vec{V}(\cdot))$ almost
  everywhere in $\Omega$. Hence the claim is a consequence of
  Lebesgue's dominated convergence theorem. 
\qed\end{proof}

The following conditions are linked to the existence of a unique
minimiser of~$\energy$.
\begin{assumption}
  \label{ass:co}
  \begin{subequations}
    We assume that $F:\Omega\times\setR_\ge\to\setR$ is strictly
    convex in its second argument, \ie it holds for all
    $\theta\in(0,1)$ and $s,t\ge 0$ with $s\not=t$, that
    \begin{align}\label{eq:convexity}
      F\big(\cdot, \theta s +(1-\theta)t\big)< \theta
      F(\cdot,s)+(1-\theta)F(\cdot,t)\qquad\text{a.e. in $\Omega$.}
    \end{align}

    Moreover, assume that $F$ satisfies the coercivity condition
    \begin{align}\label{eq:coercivity}
      F(\cdot,t)\ge g(t)\qquad\text{for
        all}~t\ge0~\text{a.e. in}~\Omega, 
    \end{align}
    with some continuous monotone function  $g:\setR_\ge\to\setR$ such
    that $g(t)\to \infty$ as $t\to \infty$.
  \end{subequations}
\end{assumption}

Now we are prepared to prove the existence of a unique solution to \eqref{eq:prob}.
\begin{proposition}\label{p:solution}
  Let $\vec{G}\in\VT^m$ and suppose Assumptions \ref{ass:Nemyckii} and
  \ref{ass:co}. 
  Then there exists a unique $\vec{U}\in\vec{G}+\VoT^m$ such that 
  \begin{align*}
    \energy(\vec{U})=\min\left\{\energy(\vec{V})\colon\vec{V}\in\vec{G}+\VoT^m\right\}.
  \end{align*}
\end{proposition}
\begin{proof}
  The proof of this assertion is fairly standard.  However, for the
  sake of a self-consistent presentation we sketch the proof here.
  Thanks to \eqref{eq:coercivity} it follows that the energy
  functional is bounded from below and that an infimal sequence
  $\{\vec{V}_k\}_{k\in\setN}\subset\vec{G}+\VoT^m$ is norm-wise
  uniformly bounded by some constant $C>0$. Since $\VT^m$ is of finite
  dimension, $\{\vec{V}_k\}_{k\in\setN}$ is pre-compact in
  $\VT^m$. Consequently, there 
  exists a converging subsequence
  $\vec{V}_{k_\ell}\to \vec{U}$ for some $\vec{U}\in\vec{G}+\VoT^m$ as
  $\ell\to\infty$.  By the continuity of $\energy$, Proposition
  \ref{prop:cont}, it follows that $\vec{U}$ is a minimiser of
  $\energy$. Finally, the strict convexity \eqref{eq:convexity}
  implies the uniqueness of the minimiser.
\qed\end{proof}

The following assumption is crucial in order to prove the discrete
maximum principle (Theorem 
\ref{t:main}). 
\begin{assumption}[Monotonicity]\label{ass:monotone}
  We assume that the mapping $F(x,\cdot):\setR_\ge\to\setR$ is
  monotone increasing for almost every $x\in\Omega$. 
\end{assumption}

Note that $\VT\subset L^\infty(\Omega)$ is finite
dimensional. Therefore, in Assumptions~\ref{ass:Nemyckii}
and~\ref{ass:co} we do not need to resort to common growth or
coercivity conditions like $\gamma(t)=c\abs{t}^s$ or $g(t)=c |t|^r$ for some
$c>0$ and $s>1$. As a consequence, the presented theory can be applied e.g. 
to the mean curvature problem; compare also with Remark~\ref{rem:meanc} below.

\section{Proof of the Discrete Convex Hull Property, Theorem \ref{t:main}}
\label{sec:main}

In the following let $K \subset \Rm$ denote a convex, closed
subset. We denote by $\Pi_K:\setR^m\to K$ the orthogonal projection
with respect to the Euclidean inner product of $\setR^m$. In other
words
\begin{align}\label{eq:PiK}
  \Pi_K x := \argmin_{y \in K} \abs{x-y}.
\end{align}

The next lemma is an immediate consequence of the convexity of
the set~$K$.
\begin{lemma}
  \label{lem:TKdir2}
  Let $K\subset\Rm$ be a closed and convex. Then 
  for all $x \in \Rm$ and $z \in K$ we have
  \begin{align*}
    (x- \Pi_K x) \cdot (z-\Pi_K x ) &\leq 0.
  \end{align*}
\end{lemma}
\begin{proof}
  If $x \in K$, then $x= \Pi_K x$ and the estimate is obvious.  Let
  us assume that $x \not\in K$. We denote by $L_x$ the hyper-plane
  which 
  touches $K$ in $\Pi_Kx$ and is orthogonal to $x - \Pi_K x$. 
  Since $K$ is convex, $L_x$ separates $K$ from $x$. This implies
  that the angle between $(\Pi_K x-x)$ and $(\Pi_K x - z)$ is greater
  or equal $\pi/2$; compare with Figure~\ref{fig:TKdir}. This proves the assertion. 
  \begin{figure}
    \centering
    \resizebox{0.6\textwidth}{!}{\input{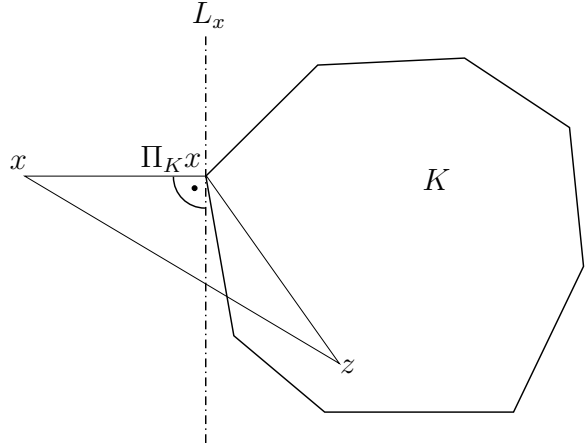}}
    \caption{The projection $\Pi_K$}
    \label{fig:TKdir}
  \end{figure}
\qed\end{proof}

We define a projection operator $\mcP_K:\VT^m\to\VT^m$ with 
$(\mcP_K\vec{V})(\Omega)\subset K$ by setting 
\begin{align}\label{df:mcP}
  (\mcP_K\vec{V})(z)\definedas \Pi_K\vec{V}(z),\qquad z\in\nodes.
\end{align}
We next list some elementary properties of this projection
operator. In particular, since $K$ is convex and $\VT^m$ contains only
piecewise linear functions over $\tria$, we have 
\begin{subequations}\label{eq:mcProp}
  \begin{align}
    \label{eq:cK}
    \mcP_K\vec{V}:\Omega\to K\qquad\text{for all}~\vec{V}\in\VT^m.
  \end{align}
  If $\vec{G}(\partial \Omega)\subset K$, for some $\vec{G}\in\VT^m$, it
  follows that the boundary values are preserved, \ie
  \begin{align}
    \label{eq:GcK}
    \mcP_K\vec{V}\in\vec{G}+\VoT^m\qquad\text{for
      all}~\vec{V}\in\vec{G}+\VoT^m.
  \end{align}
\end{subequations}
The following property of $\mcP_K$ is the key estimate in the proof of
the convex hull property.
\begin{lemma}
  \label{lem:pos}
  Let $\tria$ be a non-obtuse conforming triangulation of $\Omega$
  and let $K\subset\Rm$ be a closed and convex set. Then for
  $\vec{V}\in\VT^m$ we have
  \begin{alignat*}{2}
    \nabla \vec{V}:\nabla \mcP_K\vec{V} &\geq \abs{\nabla
      \mcP_K\vec{V}}^2 &\qquad&\text{a.\,e. in}~\Omega, 
    \intertext{and}
    \abs{\nabla \vec{V}} &\geq \abs{\nabla \mcP_K
      \vec{V}}&\qquad&\text{a.\,e. in}~\Omega.
  \end{alignat*}
\end{lemma}
\begin{proof}
  We prove the estimates on a fixed $n$-simplex $T\in\tria$. 
  Let $z_0, \dots, z_n$ denote the vertices of  $T$ 
  and let $\phi_0,
  \dots, \phi_n$ be the corresponding local Lagrange basis functions. Then
  \begin{align*}
    \vec{V}_{|T} &= \sum_{i=0}^n \vec{V}(z_i) \phi_{i|T}
  \end{align*}
  and
  \begin{align*}
    \nabla \vec{V}_{|T} : \nabla (\mcP_K \vec{V})_{|T} &= \sum_{i=0}^n \sum_{k=0}^n
    \big( \vec{V}(z_i) \cdot \Pi_K \vec{V}(z_k) \big)\, \big(\nabla
    \phi_{i|T} \cdot \nabla \phi_{k|T}\big)
  \end{align*}
  Using $\nabla \phi_0 + \nabla \phi_1 + \dots +\nabla \phi_n = \nabla
  1 = 0$ on $T$ we conclude
  \begin{align}\label{eq:VPV}
     \nabla \vec{V}_{|T} : \nabla (\mcP_K \vec{V})_{|T} &= \sum_{i=0}^n \sum_{k
      \not=i} \vec{V}(z_i) \cdot \big(\Pi_K \vec{V}(z_k) -\Pi_K
    \vec{V}(z_i) \big)\, \nabla \phi_{i|T}\cdot  \nabla \phi_{k|T}.
  \end{align}
  Lemma~\ref{lem:TKdir2} yields
  \begin{align*}
    \big( \vec{V}(z_i) - \Pi_K \vec{V}(z_i)\big) \cdot(
    \Pi_K \vec{V}(z_k)-\Pi_K \vec{V}(z_i) \big) \leq 0, 
  \end{align*}
  and thus
  \begin{align*}
    \vec{V}(z_i) \cdot \big(\Pi_K \vec{V}(z_k) -\Pi_K \vec{V}(z_i) \big)
    \leq \Pi_K \vec{V}(z_i) \cdot \big(\Pi_K \vec{V}(z_k) -\Pi_K
    \vec{V}(z_i) \big).
  \end{align*}
  Since $T$ is non-obtuse we have by Proposition \ref{p:non-obtuse} that
  $\nabla \phi_{i|T} \cdot\nabla \phi_{k|T} \leq 0$
  for all $i \neq k$. Consequently, we arrive at
  \begin{align*}
    \nabla \vec{V}_{|T} : \nabla (\mcP_K \vec{V})_{|T} &\geq \sum_{i=0}^n \sum_{k
      \neq i} \big( \Pi_K \vec{V}(z_i) \cdot(\Pi_K \vec{V}(z_k) -\Pi_K
    \vec{V}(z_i) \big)\, \big(\nabla \phi_{i|T} \cdot \nabla \phi_{k|T}\big)
    \\
    &= \sum_{i=0}^n \sum_{k=0}^n \big( \Pi_K \vec{V}(z_i) \cdot \Pi_K
    \vec{V}(z_k) \big)\, \big(\nabla \phi_{i|T} \cdot\nabla \phi_{k|T}\big)
   \\
    &= \nabla (\mcP_K \vec{V})_{|T} : \nabla (\mcP_K \vec{V})_{|T}
    =\abs{\nabla (\mcP_K \vec{V})_{|T}}^2. 
  \end{align*}
  This proves the first claim.

  Thanks to the Cauchy-Schwarz inequality, the second estimate is an
  immediate consequence of the first one. 
\qed\end{proof}

\begin{lemma}\label{lem:main}
  Let $\tria$ be a non-obtuse conforming 
  triangulation of $\Omega$ and let
  $F:\Omega\times\setR_\ge\to\setR$ satisfy Assumptions
  \ref{ass:Nemyckii}, \ref{ass:co}, and \ref{ass:monotone}.
  For $\vec{G}\in\VT^m$, $m\in\setN$, let $K\subset \setR^m$ be a
  convex and closed set such
  that $\vec{G}(\partial\Omega)\subset K$.
  Then the uniquely defined solution $\vec{U}\in \vec{G}+\VoT^m$
  of \eqref{eq:prob}
  satisfies    
  \begin{align*}
    \vec{U}(\Omega)\subset K.
  \end{align*}
\end{lemma}
\begin{proof}
  It follows from Proposition \ref{p:solution}, that problem
  \eqref{eq:prob} has a unique solution $\vec{U}\in \vec{G}+\VoT^m$.
  Since $K$ is a closed convex set we conclude from Assumption
  \ref{ass:monotone} and Lemma \ref{lem:pos} that
  \begin{align*}
    \mathcal{J}(\vec{U}) \geq \mathcal{J}(\mcP_K \vec{U}).
  \end{align*}
  It follows by \eqref{eq:GcK} that 
  $\mcP_K \vec{U} \in \vec{G}+ \VoT^m$. Since $\vec{U}$ is the unique
  minimiser of $\energy$ in  $\vec{G}+\VoT^m$ we conclude $\mcP_K \vec{U} =
  \vec{U}$. The desired assertion is then a consequence of~\eqref{eq:GcK}.
\qed\end{proof}

From this result the discrete maximum principle, Theorem
\ref{t:main}, follows as a corollary. 

\begin{proof}[Theorem \ref{t:main}]
  The convex set $K\definedas
  \operatorname{conv\,hull}\big(\vec{G}(\partial\Omega)\big)\subset
  \setR^m$ satisfies all conditions of Lemma \ref{lem:main}.
\qed\end{proof}

\begin{remark}
  We emphasize, that the proof of the convex hull property
  (Theorem~\ref{t:main}) does not 
  directly resort to Assumptions~\ref{ass:Nemyckii}
  and~\ref{ass:co}, which are only used in Proposition
  \ref{p:solution}  to ensure the existence of a
  solution. In fact, a unique minimiser $\vec{U}\in\vec{G}+\VoT$ 
  of~\eqref{eq:prob} satisfies the convex hull property
  (Theorem~\ref{t:main}) 
  if Assumption~\ref{ass:monotone} holds. The uniqueness of the
  minimiser is usually a
  consequence of the strict convexity of the energy.
\end{remark}
% \fbox{
%   \begin{minipage}{0.9\linewidth}
%     In order to have Lemma \ref{lem:pos} for each partial derivative
%     we would need $\partial_\ell\phi_i\partial_\ell \phi_k\le0$ for
%     all $i\neq k$. This is in general not true for non-obtuse meshes. 
%     To see this, we consider the following $2d$ simplex
%     \begin{align*}
%       T=\{(x,y)\colon \abs{x}\le y\le1\},
%     \end{align*}
%     \ie the triangle with corners $(-1,1)$, $(0,0)$ and $(1,1)$. The
%     local nodal basis functions on $T$ read as 
%     \begin{align*}
%       \phi_1=\frac12(y-x)\qquad \phi_2=1-y\qquad \phi_3=\frac12(x+y).
%     \end{align*}
%     Consequently 
%     \begin{align*}
%       \nabla\phi_1=
%       \frac12\begin{pmatrix}
%         -1\\1
%       \end{pmatrix}
%       \qquad
%       \nabla \phi_2=
%       \begin{pmatrix}
%         0\\-1
%       \end{pmatrix}
%       \qquad
%       \nabla\phi_3=\frac12
%       \begin{pmatrix}
%         1\\1
%       \end{pmatrix}.
%     \end{align*}
%     Obviously we have $\nabla \phi_i\cdot\nabla\phi_j\le0$, $i\neq j$,
%     but 
%     \begin{align*}
%       \partial_y\phi_1 \partial_y\phi_3=\frac14>0. 
%     \end{align*}
%   \end{minipage}
  
% }

\section{Applications and Extensions}
\label{sec:extensions}
In this section we apply the above techniques  to some particular exemplary
problems and extend Theorem \ref{t:main} in several ways.

\subsection{The \texorpdfstring{$p$}{p}-Laplace Problem}
\label{sec:SpLaplace}
This section is concerned with the so-called $p$-Laplace operator. We
verify that $p$-harmonic functions fit into the framework of Section
\ref{sec:problem-setting} and then prove a discrete maximum principle
for scalar solutions of a $p$-Laplace problem with non-positive
right-hand side.  To this end let 
$\tria$ be a conforming triangulation of the $n$-dimensional domain
$\Omega$ with polyhedral 
boundary, $\vec{G}\in\VT^m$, $m\in\setN$, and $p,q\in(1,\infty)$ with
$\frac1p+\frac1q=1$.  It is well known, that 
\begin{align}
  \label{eq:pharm}
  \begin{alignedat}{2}
    \int_\Omega\abs{\nabla \vec{U}}^{p-2}\nabla
    \vec{U}:\nabla \vec{V} \,{\rm d}x&= 0&\qquad&\text{in}~\Omega
  \end{alignedat}
\end{align}
uniquely determines a function $\vec{U}\in\vec{G}+\VoT^m$. The equivalent
minimising problem reads as
\begin{align*}
  \energy(\vec{V})\definedas \int_\Omega\frac1p \abs{\nabla\vec{V}}^p\,{\rm
    d}x\to \min\qquad \text{in}~\vec{G}+\VoT.
\end{align*}
It can easily be verified that the operator 
\begin{align*}
  F(x,t)\definedas \frac1p \abs{t}^p,\qquad t\ge0
\end{align*}
satisfies Assumptions
\ref{ass:Nemyckii}, \ref{ass:co}, and
\ref{ass:monotone}. Consequently, if $\tria$ is non-obtuse, all
conditions of Theorem \ref{t:main} are satisfied.

For $m=1$ and $f\in L^{q}(\Omega)$ we investigate the non-homogeneous
$p$-Laplace problem 
\begin{align}\label{eq:pLaplace}
  \begin{alignedat}{2}
    \int_\Omega\abs{\nabla U}^{p-2}\nabla
    U\cdot\nabla V \,{\rm d}x&= \int_\Omega fV\,{\rm d}x&\qquad&\text{in}~\Omega\\
    U&=G&\qquad&\text{on}~\partial\Omega.
  \end{alignedat}
\end{align}
The unique solution $U\in G+\VoT$ is the
minimiser of the energy
\begin{align*}
  \energy_f(V)\definedas \int_\Omega\frac1p \abs{\nabla V}^p-fV\,{\rm d}x,
  \qquad V\in G+\VoT.
\end{align*}
Since the integrand depends not only on the gradient of the function
but also on the function value itself, we can not directly apply Theorem
\ref{t:main}. However, by modifying the arguments we can prove the
following result. 
\begin{theorem}\label{t:scalar-main}
  Let the conforming triangulation $\tria$ of $\Omega$ be non-obtuse
  and let $f\le0$ in $\Omega$. 
  
  Then the unique solution $U\in G+\VoT$ of \eqref{eq:pLaplace} satisfies
  \begin{align}
    \max U(\Omega)\le \max U(\partial\Omega). 
    \tag{DMP}
 \end{align}
\end{theorem}

\begin{proof}
  We define the closed convex set
  \begin{align*}
    K\definedas
    \big(-\infty,
      \,\max U(\partial\Omega)\big]
    \supset\operatorname{conv\,hull}\big(U(\partial\Omega)\big). 
  \end{align*}
  Hence the projection operator $\mcP_K:\VT\to\VT$, defined in
  \eqref{df:mcP} satisfies \eqref{eq:mcProp}, \ie we have
  that $\mcP_KU\in G+\VoT$. Clearly the mapping $(x,t)\mapsto
  F(x,t)\definedas \frac1pt^p$, $(x,t)\in\Omega\times\setR_\ge$,
  satisfies all conditions in Section 
  \ref{sec:energy-minimisation}. In particular, it satisfies
  Assumption \ref{ass:monotone}, \ie it is
  monotone increasing in its second argument. 
  Therefore, we obtain by Lemma \ref{lem:pos} that 
  \begin{align*}
    \int_\Omega\frac1p\abs{\nabla \mcP_K U}^p\,{\rm d}x\le
    \int_\Omega\frac1p\abs{\nabla U}^p\,{\rm d}x. 
  \end{align*}
  Moreover, by means of \eqref{eq:PiK} and the particular choice of 
  the convex set $K$, we have 
  \begin{align*}
    \mcP_K U(x)= \min\big\{U(x), \max U(\partial\Omega)\big\}\le
    U(x),\quad\text{for all}~x\in\Omega.
  \end{align*}
  Hence the assumption $f\le 0$ in $\Omega$ implies 
  \begin{align}\label{eq:f}
     \int_\Omega f\mcP_KU \,{\rm d}x\ge\int_\Omega fU \,{\rm d}x.
  \end{align}
  Combining these estimates yields
  \begin{align*}
    \energy(\mcP_K U)\le \energy(U).
  \end{align*}
  Since $U$ is the unique minimiser we arrive with \eqref{eq:mcProp} at 
  $U(\Omega)= (\mcP_K U)(\Omega)\subset K$. This proves the
  assertion. 
\qed\end{proof}

\begin{remark}[Orlicz functions]
  \label{rem:psi}
  The $p$-Laplace operator often serves as a prototype for a much
  larger class of non-linear operators: Indeed, if we have a continuous,
  strictly convex and monotone function $\psi:\setR_\ge\to\setR_\ge$,
  then the non-linear function $F(x,t)\definedas \psi(t)$ satisfies
  Assumptions \ref{ass:Nemyckii}, \ref{ass:co}, and
  \ref{ass:monotone}. Consequently, the discrete convex hull property 
  (Theorem \ref{t:main}) holds also for this
  larger class of problems. 
  
  In order to prove Theorem \ref{t:scalar-main} for more
  general nonlinearities we need to ensure the solvability
  of the problem. This can e.g. be achieved either by restrictions 
  on the right-hand side
  $f$ or by assuming the additional growth condition 
  $\lim_{t\to\infty}\psi(t)/t=\infty$ for the non-linearity. 
  We illustrate the first case in Remark \ref{rem:meanc} below by means of
  the mean curvature problem, \ie  $\psi(t) = \sqrt{1+t^2}$.
  The latter case includes 
  for example all N-functions.
  For more information on N-functions and on finite element
  approximations for this kind of problems we refer to
  \cite{DieningRuzicka:07b,DieningKreuzer:08,BelDieKreu:11}.
\end{remark}
\begin{remark}[Mean curvature problem]
  \label{rem:meanc}
  We emphasize that also the discrete mean curvature
  problem fits into the setting of
  Section~\ref{sec:energy-minimisation}. To see this, we consider the
  corresponding energy 
  \begin{align*}
    \energy(\vec{V})\definedas 
    \int_{\Omega}\sqrt{1+\abs{\nabla{\vec{V}}}^2}\,{\rm d}x, \quad
    \vec{V}\in \VT^m.
  \end{align*}
  Obviously, the mapping
  $\Omega\times\setR_\ge\ni(x,t)\mapsto F(x,t)\definedas
  \sqrt{1+t^2}$ has linear growth and is monotone in $t$. Therefore,
  it satisfies Assumptions 
  \ref{ass:Nemyckii} and \ref{ass:monotone}. 
  Moreover, it is coercive and thanks to $\tfrac{d^2}{dt^2}
  \sqrt{1+t^2}=\frac1{\sqrt{1+t^2}}>0$ for all $t\ge0$, it is strictly convex. 
  This verifies Assumption~\ref{ass:co} and  consequently, minimisers
  $\vec{U}\in\vec{G}+\VoT^m$ of the 
  energy functional $\energy$ in $\vec{G}+\VoT^m$, $\vec{G}\in\VT$, 
  satisfy the convex hull property
  Theorem~\ref{t:main}. 
  
  Consider the scalar case $m=1$ and let $f\in L^2(\Omega)$, $f\le0$
  in $\Omega$.  Let $G\in\VT$, 
  according to \cite{FierroVeeser:03} the energy 
  \begin{align}\label{eq:en_meanc}
    \energy_f({V})\definedas 
    \int_{\Omega}\sqrt{1+\abs{\nabla{V}}^2}-fV\,{\rm d}x, \quad V\in G+\VoT,
  \end{align}
  is coercive if and only if there exists $\varepsilon>0$, such that 
  \begin{align*}
    \int_\Omega fV\,{\rm d}x\le (1-\varepsilon) \int_\Omega\abs{\nabla
      V}\,{\rm d}x\qquad\text{for all}~V\in\VoT.
  \end{align*}
  Under this condition Theorem \ref{t:scalar-main} directly applies to
  minimisers of the energy defined in \eqref{eq:en_meanc}. 
\end{remark}

\subsection{Problems Involving Lower Order Terms}
\label{sec:react-diff}
In this section we investigate
non-linear problems involving lower order terms in form of lumped
masses.  
Let $\tria$ be a conforming triangulation of $\Omega$ and
$\vec{G}\in\VT^m$, $m\in\setN$. For $1<p,q<\infty$ let the energy
\begin{align}\label{eq:energy_rd}
  \energy(\vec{V})\definedas\int_\Omega\tfrac1p\abs{\nabla \vec{V}}^p+\tfrac1q\,
  \abs{\vec{V}}^q{\rm d}x\qquad\vec{V}\in\VT^m.
\end{align}
serve as our prototype. Obviously, there exists a unique
minimiser of $\energy$ in $\vec{G}+\VoT^m$.

In order to prove a convex hull property (Theorem \ref{t:main}) in the
spirit of Section~\ref{sec:main}, we would need an analogue of Lemma
\ref{lem:pos}  for function values. 
Unfortunately, in general  we do not have 
  $\abs{\vec{V}}\ge\abs{\mcP_K \vec{V}}$ point-wise in $\Omega$.
% \end{align*}
% for compact sets $K\subset\setR^m$. 
One can construct examples where the discrete maximum principle is violated even
in one dimension; see e.g. \cite{BrandtsKorotovKrizek:2008}. 
However, for $p=q=2$ it is well known that maximum principles hold for
sufficiently small mesh-sizes if
the mesh is strictly acute; see
e.g. \cite{CiarletRaviart:73,BrandtsKorotovKrizek:2008}.

This additional assumptions can be avoided by mass-lumping, \ie instead of a
  minimiser  of~\eqref{eq:energy_rd} we look for a minimiser of the
  modified energy 
  \begin{align}\label{eq:energy_lump}
    \energy_\ell(\vec{V})\definedas \int_\Omega \frac1p\abs{\nabla
      \vec{V}}^p\,{\rm d}x+\sum_{z\in\nodes} 
    \frac{\abs{\support(\phi_z)}}{n+1} \frac 1q \abs{\vec{V}(z)}^q,
  \end{align}
  which is coercive and
  strictly convex. Hence there exists a unique $\vec{U}\in {G}+\VoT^m$
  minimising $\energy_\ell$ in $\vec{G}+\VoT^m$. This minimiser
  satisfies the following convex hull property.

\begin{theorem}\label{thm:rdlumped}
    Let the conforming triangulation $\tria$ of $\Omega$ be non-obtuse
    and let
    $\vec{U}\in \vec{G}+\VoT^m$ be the unique minimiser of  the energy
    $\energy_\ell$ in $\vec{G}+\VoT^m$.
    
    Then we have the max hull property
    \begin{align*}
      \vec{U}(\Omega)\subset\operatorname{conv\, hull}\big(
      \vec{U}(\partial\Omega)\cup \{0\}\big). 
    \end{align*}
\end{theorem}
\begin{proof}
  Thanks to $K\definedas \operatorname{conv\, hull}\big(
  \vec{G}(\partial\Omega)\cup \{0\}\big)$ we can choose $z=0$ in
  Lemma \ref{lem:TKdir2} and observe that 
  \begin{align*}
    0\ge \abs{\Pi_K x}^2-x\cdot\Pi_Kx \ge \abs{\Pi_K
      x}^2-\abs{x}\abs{\Pi_Kx}\quad\text{for all}~x\in \setR^m.
  \end{align*}
  Consequently we have
  \begin{align*}
    \abs{\Pi_K\vec{U}(z)}\le
    \abs{\vec{U}(z)}\qquad\text{for all}~z\in\nodes.
  \end{align*}
  Combining this with Lemma \ref{lem:pos}, we obtain
  \begin{align*}
    \energy_\ell(\mcP_K \vec{U})\le \energy_\ell(\vec{U}).
  \end{align*}
  Recalling that the minimiser $\vec{U}$ is unique in
  $\vec{G}+\VoT^m$ we have $\vec{U}=\mcP_K\vec{U}$. Hence
  we can conclude with  \eqref{eq:cK} that 
  $\vec{U}(\Omega)= (\mcP_K \vec{U})(\Omega)\subset
  K$. This proves the assertion. 
\qed\end{proof}

\begin{remark}
  In the scalar case $m=1$ the minimiser $U\in G+\VoT$ of the energy 
  \begin{align*}
    \energy_\ell({V})\definedas \int_\Omega \frac1p\abs{\nabla
      {V}}^p-fV\,{\rm d}x+\sum_{z\in\nodes} 
    \frac{\abs{\support(\phi_z)}}{n+1} \frac 1q \abs{{V}(z)}^q,
  \end{align*}
  with $f\le0$ satisfies the discrete maximum principle
  \begin{align*}
    \max U(\Omega)\le \max\left\{0,\,\max U(\partial\Omega)\right\}. 
  \end{align*}
  This claim is an easy consequence of the above arguments combined
  with \eqref{eq:f}.  
\end{remark}

  \subsection{Strong Convex Hull Property}
\label{sec:strict}

In some situations we have a strong version of the
convex hull property. For scalar functions this is well known as 
the strong maximum principle: if a subharmonic
function attains a global maximum in the interior of the domain
$\Omega$, then the functions must be constant (on the connected component).

We will introduce in this section a discrete strong convex hull
property. 
To this end we need that the conforming triangulation
$\mathcal{T}$ of $\Omega$ is acute:
A simplex $T\in\tria$ is called acute if the angles between
any two of its sides are less than $\frac\pi2$. Similar to
Definition~\ref{df:non-obtuse}, we can formulate this definition also in
terms of the Lagrange basis.
\begin{definition}[acuteness]\label{df:acute}
  Let $\tria$ be a conforming triangulation of $\Omega$. 
  We call an $n$-simplex $T\in\tria$ 
    acute if 
  \begin{align*}
    \nabla \phi_{z|T}\cdot\nabla\phi_{y|T}< 0 \qquad\text{for
      all}~z,y\in T\cap\nodes~\text{with}~z\neq y.
  \end{align*}

  The conforming triangulation $\tria$ of $\Omega$ is called
    acute if all simplexes $T\in\tria$ are acute.
\end{definition}

Let $K \subset \Rm$ be a non-empty, convex, and closed set. Then we call
$x \in K$ an extreme point if it is no element of a line
segment of two different points of~$K$. By $\extr(K)$ we denote the set of
extreme points of~$K$. It is well known that if $K$ is additionally
bounded, we have
$\convexhull(\extr(K)) =\convexhull(K)$.

We assume the following additional condition on our mesh.
\begin{assumption}\label{ass:TcapNo}
  We assume that $\tria$ is a conforming triangulation of $\Omega$ such
  that for every $T\in\tria$, we have $T\cap\nodes_0\neq\emptyset$. 
\end{assumption}
  In other words, every simplex of the triangulation $\tria$ has at least one interior
  vertex. As a consequence, we have that 
  every vertex of the boundary is in the support of one of
  the interior Lagrange basis functions. Assumption \ref{ass:TcapNo} is not a
  severe restriction.

We can now state the discrete strong convex hull property for
discrete $p$-harmonic functions, \ie solutions of \eqref{eq:pharm}. 
\begin{theorem}
  \label{thm:strict}
  Let the acute conforming triangulation $\tria$ of $\Omega$ satisfy
  Assumption \ref{ass:TcapNo}. Let $\vec{U}\in \vec{G}+\VoT^m,$ be
  discretely $p$-harmonic 
  with boundary values~$\vec{G}\in\VT^m$, \ie $\vec{U}$ is a
  solution of~\eqref{eq:pharm}. 

  If $\vec{U}(z_0) \in
  \extr(\convexhull(\vec{U}(\Omega))$ for some $z_0 \in \Omega$,
  then $\vec{U}$ is constant.
\end{theorem}

Before we get to the proof of Theorem~\ref{thm:strict} we need a local
version of this result.
\begin{lemma}
  \label{lem:strictcell}
  Let $\tria$ be acute and let $\vec{U}$ be a solution
  of~\eqref{eq:pharm}.  Let $z_0 \in \nodes_0$ and let $C:= (\nodes
  \setminus \set{z_0}) \cap \support \phi_{z_0}$, \ie $C$ is the set
  of neighbors of~$z_0$. Then $\vec{U}(z_0) \in \extr( \convexhull(
  \vec{U}(C)))$ implies that $\vec{U}$ is constant on $\support
  \phi_{z_0}$.
\end{lemma}
\begin{proof}
  Let us assume that $\vec{U}$ is non-constant on $\support
  \phi_{z_0}$. We have to prove that $\vec{U}(z_0) \not\in \extr(
  \convexhull( \vec{U}(C)))$.

  Since $z_0\in\nodes_0$, we have for every unit vector
  $\vec{e}_j\in\setR^m$, $j=1, \dots,m$, that
  $\vec{e}_j \phi_{z_0} \in 
  \VoT^m$. Therefore, it follows from~\eqref{eq:pharm} that 
  \begin{align*}
    \int_\Omega \abs{\nabla \vec{U}}^{p-2} \nabla\vec{U}:\nabla
    (\vec{e}_j \phi_{z_0}) \,{\rm d}x &= 0 \qquad \text{for all $j=1,\dots,m$}.
  \end{align*}
  Thanks to the representation $\vec{U} = \sum_{y \in \nodes} \vec{U}(y) \phi_y$,
   we thus get
  \begin{align*}
    \sum_{y \in \nodes} \vec{U}(y) \int_\Omega \abs{\nabla
      \vec{U}}^{p-2} \nabla \phi_y \cdot\nabla \phi_{z_0} \,{\rm d}x=0
  \end{align*}
  or equivalently 
  \begin{align*}
    \vec{U}(z_0) \underbrace{\int_\Omega \abs{\nabla
        \vec{U}}^{p-2} \nabla \phi_{z_0} \cdot\nabla \phi_{z_0} \,{\rm d}x}_{:=
      \beta_0} &= \sum_{y \in C} \vec{U}(y) \underbrace{\int_\Omega
      \abs{\nabla \vec{U}}^{p-2} (-\nabla \phi_y \cdot\nabla \phi_{z_0})
      \,{\rm d}x}_{=:\beta_y}.
  \end{align*}
  By assumption $\vec{U}$ is non-constant on $\support \phi_{z_0}$, hence
  $\beta_0 > 0$. On the other hand, since $\tria$ is acute, we have $\beta_y \geq 0$
  for all $y \in C$. Let $T\in \tria$, $T\subset\support\phi_{z_0}$
  such that $\vec{U}_{|T}\not\equiv0$.
% Moreover, since $\vec{U}$ is non-constant on
%   $\support\phi_{z_0}$, we conclude 
%   that $\nabla \vec{U}$ is non-zero on at least one simplex~$T$
%   with vertex~$z_0$.
  If $y\in\nodes\cap T$,~$y\neq z_0$, then the
  acuteness of~$T$ implies $\beta_y>0$. Thus at least two (more
  precisely: at least~$n$) of the $\beta_y$ with $y \in C$
  satisfy~$\beta_y>0$.

  The identity $\phi_{z_0} + \sum_{y \in C} \phi_y = 1$ on
  $\support \phi_{z_0}$ implies $\beta_{0} =\sum_{y \in C} \beta_y$.
  Therefore, defining $\lambda_y := \beta_y/\beta_{z_0}$ yields $\vec{U}(z_0) =
  \sum_{y \in C} \lambda_y \vec{U}(y)$ with $\lambda_y \geq 0$ for $y \in C$
  and $\sum_{y \in C} \lambda_y = 1$. Consequently, at least two of the
  $\lambda_y$ with $z \in C$ satisfy~$\lambda_y>0$ and hence $\vec{U}(z_0)$
  cannot be an extreme point of $\convexhull \set{\vec{U}(y) \,:\, y \in
    C}$. In other words, we have  $\vec{U}(z_0) \not\in \extr(
  \convexhull( \vec{U}(C)))$.
\qed\end{proof}
We can now prove Theorem~\ref{thm:strict}.
\begin{proof}[Theorem~\ref{thm:strict}]
  Since $\vec{U}$ is piecewise linear and thanks to the properties of
  extremal points, it follows that there exists
  $\tilde z_0\in\nodes_0$ with $\vec{U}(z_0)=\vec{U}(\tilde
  z_0)$. Therefore, we may assume w.l.o.g. that $z_0\in\nodes_0$.
 
  Suppose $\vec{U}(z_0) \in \extr(\convexhull(\vec{U}(\Omega))$ and 
  let $C\definedas (\nodes \setminus
  \set{z_0}) \cap \support \phi_{z_0}$ be the neighbors of~$z_0$. Due
  to the convex hull 
  property of $\vec{U}$ (see Theorem~\ref{t:main}) we have that
  $\convexhull (\vec{U}(C)) \subset \convexhull( \vec{U}(\partial
  \Omega))$. Since $\vec{U}(z_0)$ is an extreme point of $\convexhull(
  \vec{U}(\partial \Omega))$ it must also be an extreme point of
  $\extr(\convexhull(\vec{U}(C))$. Therefore, by Lemma~\ref{lem:strictcell},
  $\vec{U}$ is locally constant on~$\support \phi_{z_0}$. 

  A repetition
  of this argument shows that $\vec{U}$ is constant on its connected
  component, which is~$\Omega$ itself. For this argument we need, that every
  simplex contains at least one vertex of~$\nodes_0$.
\qed\end{proof} 

\begin{remark}
  We emphasize that \eqref{eq:pharm} serves as a model problem and that
  Theorem \ref{thm:strict} applies to more general non-linearities. For example,
  the strong convex hull property applies to the
  non-linear problems mentioned in 
  Remark~\ref{rem:psi} if $\mathcal{J}$ is additionally \Frechet{}
  differentiable, \ie if the minimising problem \ref{eq:prob} can be equivalently
  formulated by a partial differential equation, the
  so-called Euler-Lagrange equation; compare with~\eqref{eq:DJ}. 
\end{remark}

\section*{Acknowledgements:} 
Part of this work was carried out during a stay of Christian Kreuzer
at the Mathematical Institute of the University of Oxford. 
This stay was financed by the German Research Foundation
DFG within the research grant KR 3984/1-1.

Last but not least, the authors want to thank Adrian Hirn for the valuable and
inspiring discussions during his stay in Munich.

\providecommand{\bysame}{\leavevmode\hbox to3em{\hrulefill}\thinspace}
\providecommand{\MR}{\relax\ifhmode\unskip\space\fi MR }
% \MRhref is called by the amsart/book/proc definition of \MR.
\providecommand{\MRhref}[2]{%
  \href{http://www.ams.org/mathscinet-getitem?mr=#1}{#2}
}
\providecommand{\href}[2]{#2}

% \bibliographystyle{amsalpha}

% \bibliography{maxhull}

\end{document}